\documentclass[letterpaper]{amsart}
\usepackage{amsthm}
\usepackage{amssymb,latexsym,graphics,enumerate}
\usepackage[mathscr]{eucal}
\usepackage{amsmath,amsfonts,amsthm,amssymb}
\usepackage{color}
\usepackage{hyperref}
\usepackage[left=1in,right=1in,top=1in,bottom=1in]{geometry}
\usepackage[abbrev,lite,nobysame]{amsrefs}

\numberwithin{equation}{section}

\newcommand{\lan}{\langle}
\newcommand{\ran}{\rangle}
\newcommand{\be}{\begin{eqnarray*}}
\newcommand{\bel}{\begin{eqnarray}}
\newcommand{\ee}{\end{eqnarray*}}
\newcommand{\eel}{\end{eqnarray}}
\newcommand{\ba}{\begin{aligned}}
\newcommand{\ea}{\end{aligned}}
\newcommand{\de}{\Delta}

\newcommand{\pa}{\partial}

\newcommand{\pn}{\phi_{\neq}}
\newcommand{\pz}{\lan \phi\ran}

\newtheorem{theorem}{Theorem}[section]
\newtheorem{lemma}[theorem]{Lemma}

\newtheorem{proposition}[theorem]{Proposition}
\newtheorem{corollary}[theorem]{Corollary}
\theoremstyle{definition}
\newtheorem{assumption}{Assumption}[section]
\theoremstyle{definition}
\newtheorem{example}{Example}[section]
\newtheorem{remark}{Remark}[section]

\newcommand{\gn}{g_{\neq}}
\newcommand{\gz}{\lan g\ran}

\newcommand{\norm}[1]{\left\lVert#1\right\rVert}

\newcommand\N{{\mathbb N}}
\newcommand\R{{\mathbb R}}
\newcommand\T{{\mathbb T}}

\newcommand{\cS}{\mathcal{S}}

\newcommand{\ZZ}{\mathbb{Z}}

\newcommand{\TT}{\mathbb{T}}

\newcommand{\bv}{\mathbf{v}}

\newcommand{\rL}{\mathring{L}}

\newcommand{\op}{\text{op}}

\usepackage{verbatim}
\usepackage{esint}
\usepackage{bm}
\usepackage{enumitem}
\usepackage{amsmath,amsthm}

\usepackage{color}

\def\eps{\varepsilon}
\def\e{{\rm e}}

\def\dd{{\rm d}}
\def\sign{{\rm sign}}
\def\ddt{{\frac{\dd}{\dd t}}}
\def\R {\mathbb{R}}
\def\ZZ {\mathbb{Z}}
\def\Re{{\rm Re}}
\def\Im{{\rm Im}}
\def \l {\langle}
\def \r {\rangle}
\def\T {{\mathbb T}}
\def\de{{\partial}}

\def\N{{\mathbb{N}}}

\begin{document}

\title[Kuramoto-Sivashinsky with Shear]{Global existence for the two-dimensional Kuramoto-Sivashinsky equation with a shear flow}

\author[M. Coti Zelati {\em et Al.}]{Michele Coti Zelati}
\email{m.coti-zelati@imperial.ac.uk}
\author[]{Michele Dolce}
\email{m.dolce@imperial.ac.uk}
\address{Department of Mathematics, Imperial College,
South Kensington Campus,
London SW7 2AZ, UK}

\author[]{Yuanyuan Feng}
\email{yzf58@psu.edu}
\author[]{Anna L. Mazzucato$^\ast$}
\email{alm24@psu.edu}
\address{Department of Mathematics, Penn State University, University Park, PA
16802, USA}
\thanks{$\ast$ Corresponding author.  On behalf of all authors, the corresponding author states that there is no conflict of interest.
}
\thanks{Data sharing not applicable to this article as no datasets were generated or analyzed during the current study.
}

\begin{abstract}
We consider the Kuramoto-Sivashinsky equation (KSE) on the two-dimensional torus in the presence of advection by a given background shear flow. Under the assumption that the shear has a finite number of critical points and there are linearly growing modes only in the direction of the shear,  we prove global existence of solutions with data in $L^2$,
using a bootstrap argument. The initial data can be taken arbitrarily large.
\end{abstract}

\keywords{Two dimension, Kuramoto-Sivashinsky, mixing, global
existence, mild solutions, enhanced diffusion, diffusion time}

\subjclass[2010]{35K25, 35K58, 76E06, 76F25}

\maketitle

\section{Introduction} \label{s:intro}
In this article we consider the Kuramoto-Sivashinsky equation (KSE) in two-space dimension in the presence of advection by a given background shear flow. The KSE is a well-known model of large-scale instabilities, such as those arising in flame-front propagation (see e.g \cite{HN86} and references therein). Since KSE describes the motion of an interface  in a suitable coordinate system, the physically relevant dimensions are 1D and 2D.

The KSE comes in a scalar, potential form, and a differentiated, vectorial form. We will confine ourselves to the scalar form, since the addition of a linear transport term is meaningful for the potential:
\begin{align} \label{eq:KSE}
\partial_t \phi +\frac{1}{2}|\nabla\phi|^2+\Delta^2\phi+\Delta\phi=0\,.
\end{align}
We solve this equation with periodic boundary conditions on $[0, L_1]\times[0, L_2]$, that is, on a two-dimensional torus, which with slight abuse of notation we denote by $\TT^2$. When $L_1>2\pi$ or $L_2>2\pi$, the symbol of the linear operator $\Delta^2+\Delta$ is negative on a finite set of low frequencies.  Hence there are (linearly) growing modes in the horizontal or vertical direction respectively.

We consider a modified version  of \eqref{eq:KSE}, where the potential $\phi$ is subject to advection by a given steady shear flow, which we write without loss of generality as  the horizontal shear ${\bv}=(u(y),0)$:
\begin{align*} 
 \partial_t \phi+A\,u(y)\partial_x \phi+\frac{1}{2}|\nabla\phi|^2+\Delta^2\phi+\Delta\phi=0,
\end{align*}
where the parameter $A>0$ represents the amplitude of the flow.  The KSE with general advection term has been utilized in models of turbulent premixed-combustion \cite{DK95}. By a change of time, the above equation can be rewritten in an equivalent way as
\begin{align} \label{eq:AKSE}
 \partial_t \phi+u(y)\partial_x \phi+\frac{\nu}{2}|\nabla\phi|^2+\nu \Delta^2\phi+\nu \Delta\phi=0,
\end{align}
 where $\nu=A^{-1}$ and, with slight abuse of notation, we have not relabeled the transformed variables. Because $\nu$ determines the strength of the dissipation, we will refer to $\nu$ as a viscosity coefficient. We will refer to the equation above as AKSE.

The main difficulty in dealing with both \eqref{eq:KSE} and \eqref{eq:AKSE} is the lack of {\em a priori} norm estimates on the solution, which does not allow to bootstrap local existence into global existence via  a standard continuation argument.  The analysis of the Kuramoto-Sivashinsky equations in one space dimension is well developed by now, since in one dimension energy estimates lead to a good control on the $L^2$ norm of the solution \cites{BG06,CEES93,Gru00,GF19,GJO15,GO05,Otto09} . By contrast, there are only a handful of results concerning the well-posedness of the classical KSE \eqref{eq:KSE} in dimension greater than one. Local well-posedness holds in $L^p$ spaces \cites{BS07,IS16}. Global existence is known only under fairly restrictive assumptions, such as for thin domains and for the anisotropically reduced KSE \cites{BKRZ14, KM21, LY20,SellT92}, without growing modes \cites{AM19,FM20}, or with only one growing mode in each direction \cite{AM21}, for small data.

In \cite{FM20}, two of the authors proved global existence for AKSE for large data  and any number of growing modes, when the advecting velocity field induces a sufficiently small dissipation time, e.g. if the flow is mixing, that leads to a global uniform bound on the $L^2$ norm of the solution. In this case, the action of the flow is to move energy from large scales to small scales in both directions, where the dissipation can efficiently damp  the effect of all the growing modes. We prove in this work that the same result, global existence for large data, holds, if the advecting flow is a shear flow with only {\em isolated} critical points and in the absence of growing modes in the direction transverse to the shear, the vertical direction in our set up, which can be achieved by restricting $0<L_2<2\pi$.
The idea of the proof is to exploit the enhanced dissipation arising from the combined action of the hyper-diffusion and the advection to control both the non-linearity as well as the destabilizing effect of the negative Laplacean at large scale.
Intuitively speaking, the shear flow has no influence on purely vertical modes. For instance, the function of $y$ obtained by averaging the solution in the $x$ direction may grow in time. On the other hand, the
mixing along  streamlines of the flow moves energy from large to small scales. Therefore, the growth generated by growing horizontal modes is damped on a sufficiently large time-scale by the dissipation. The non-linearity then couples all the modes.

For the case at hand of a steady shear flow, the transport operator has a large kernel, namely all the functions on the torus that are constant in the horizontal variable. One needs to project out the kernel to take advantage of the action of the flow. There is no enhanced decay of the energy on the kernel component (at a linear level), but the norm can nevertheless be controlled as they satisfy collectively a modified one-dimensional KSE.
A key point is to use the fact that the linear operator
\begin{equation} \label{eq:HtildenuDef}
	H_\nu:= \nu\Delta^2+u(y)\pa_x,
\end{equation}
is dissipation enhancing \cites{CKRZ08,CZDE20,FI19}. More precisely, for the components of the solution orthogonal to the kernel of the transport operator, it generates an exponentially stable semigroup $\e^{-tH_{\nu}}$ with a rate of decay of the $L^2$ norm of order $\lambda_\nu$, where $\nu/\lambda_\nu\to 0$ as $\nu \to 0$. By contrast,  a standard  energy estimate shows that the semigroup is contractive with rate $O(\nu)$. The improved rate in viscosity allows to control both the growing modes as well as the nonlinear terms, provided $\nu$ is small enough compared to the size of the initial data.
Given $g\in L^2(\TT^2)$, we denote 
\begin{equation}
	\gz(y)=\frac{1}{L_1}\int_{\T^1}g(t,x,y) \dd x, \qquad  \gn(x,y)=g(x,y)-\gz(y).
\end{equation} 
By Fubini-Tonelli's Theorem, $\gz$ exists for a.e. $y$.
We observe that $\lan g\ran$ corresponds to the projection of  $g$ onto the kernel of the advection operator $u(y) \pa_x$, while $ g_{\neq}$ corresponds to the projection onto the orthogonal complement in $L^2$. As shown in \cite{BCZ17}, if $u$ has a finite number of 
critical points of order at most $m\geq 2$, namely at most $m-1$ derivatives vanish at the critical points, then $u$ is mixing in the sense that
for some  constant $C>0$ there holds
\begin{align}
\| \e^{-u\de_x t} \gn\|_{H^{-1}}\leq \frac{C}{(1+t)^m}\| \gn\|_{H^{1}},
\end{align} 
for every $t\geq 0$. Thanks to \cite{CZDE20}*{Corollary 2.3}, this translates into the enhanced dissipation estimate
\begin{align}\label{bd:semest}
\| \e^{-H_\nu t} \gn\|_{L^2}\leq 5 \e^{-\lambda_{\nu} t}\| \gn\|_{L^2}, \qquad \lambda_\nu=\eps_0 \nu^{\frac{3m}{3m+2}},
\end{align} 
 for some $\eps_0>0$, independent of $\nu$, and for every $t\geq 0$.
For AKSE, we use \eqref{bd:semest} to show that solutions are global, as stated in the next theorem.

\begin{theorem}
\label{t:main}
Let  $0<L_2<2\pi$  and let  $u:[0,L_2)\to \R$ be a smooth function with a finite number of critical points of order at most $m\geq 2$. Then, given $\phi_0\in L^2(\TT^2)$,   there exists $0<\nu_0<1$ depending on $L_1,\,L_2, u$, and $\|\phi_0\|_{L^2}$ with the following property: for any $0<\nu<\nu_0$, there exists a global-in-time weak solution $\phi$ of \eqref{eq:AKSE} with initial data $\phi_0$ such that $\phi\in L^\infty([0,T),L^2)\cap L^2([0,T),H^2)$ for all $0<T<\infty$.
\end{theorem}

We stress that we can allow {\em any} number of growing modes in the horizontal direction and the initial data can be {\em arbitrarily large} in $L^2$.

The proof is based on a bootstrap argument inspired by \cite{BH17}.  The main steps in this argument are as follows. For any initial data in $L^2$,  there exists a local-in-time mild solution of  \eqref{eq:AKSE}  on some interval of time $[0,t_0)$, which is also a weak solution in $L^\infty([0,t_0),L^2)\cap L^2([0,t_0),H^2)$ and satisfies the energy identity \cite{FM20}. For $t_0$ small enough, we can make the $L^2$ and $H^2$ norms of the projected component less than a certain multiple of the size of the initial data. By using the stability of the semigroup generated by $H_\nu$, one then shows that, for $\nu$ sufficiently small, these norms are in fact half that amount. Hence the solution can be continued for a longer time than $t_0$, which allows to bootstrap existence from local to global for the projected component and then conclude using the time evolution of the kernel component of the solution.

As we shall see in Section \ref{s:global}, the size of $\nu_0$ in Theorem \ref{t:main} depends on the rate at which $\nu/\lambda_\nu$ vanishes
as $\nu\to 0$. Hence, improving the semigroup estimate \eqref{bd:semest} automatically implies a better global existence threshold.
In Section \ref{s:semigroup}, we show that imposing a possibly more restrictive condition on $u$ (see Assumption \ref{a:lowerbddE}),  the semigroup bound can be improved.
 In particular, we consider as a prototypical example the case of
 \begin{equation}
 		\label{def:sinym}
 		u(y)=\sin((2\pi y)/L_2)^m \qquad \text{for } m\in \mathbb{N},
 	\end{equation}
and prove the following result.
\begin{proposition} \label{p:semest}
Let $g\in L^2(\TT^2)$, $0<\nu<1$, and $u(y)$ be given as in \eqref{def:sinym}. There exists $\eps'_0>0$, independent of $\nu$, such that
\begin{equation}\label{bd:semest2}
	\|\e^{-tH_\nu }\gn\|_{L^2}\leq \e^{-\lambda'_{\nu} t+\pi/2} \lVert \gn\rVert_{L^2}, \qquad \lambda'_\nu=\eps'_0 \nu^{\frac{\max\{2,m\}}{\max\{2,m\}+4}},
\end{equation}
for every $t\geq 0$.
\end{proposition}
Notice that the role of $m\in \N$ here is precisely that of \eqref{bd:semest}, as $u$ in \eqref{def:sinym} has critical points of order 
at most $\max\{2,m\}$. Hence, a direct comparison between  \eqref{bd:semest} and \eqref{bd:semest2} shows that
\eqref{bd:semest2} has a much better decay rate, and in particular $\nu/\lambda'_\nu\to 0$ faster as $\nu\to0$.

The derivation of the semigroup estimate \eqref{bd:semest2} is carried out in Section \ref{s:semigroup} via a spectral-theoretic approach.  It follows from a general Gearhart-Pr\"uss criterion for m-accretive operators devised in \cite{Wei18} based on a quantitative pseudo-spectral
bound. The proof is motivated by that of a similar result for the Laplace operator $\Delta$ in \cite{gallay19}. For the Laplace operator plus advection,  decay rates  akin to \eqref{bd:semest2} were obtained in \cites{CCZW20,Wei18} for a shear with infinitely many critical points, using the pseudo-spectral approach,  and for shear flows with finitely many critical points in \cite{BCZ17}, using hypocoercivity.
Such quantitative semigroup estimates are relevant  in the investigation of enhanced diffusion for passive scalars \cites{BW13,BCZ17,CD20,Wei18},  in the study of
asymptotic stability of particular solutions to the two-dimensional Navier-Stokes equations \cites{CZEW20,Gal18,LWZ20,WZZ20}, and have also applications to
several other nonlinear problems \cites{BH17, HT19, IXZArXiv19, KX16}.

In Section \ref{s:semigroup}, we prove a more general version of Proposition \ref{p:semest}, namely Proposition \ref{t:SemigroupEst}, for shear flows satisfying a certain condition, Assumption \ref{a:lowerbddE}, again inspired by \cite{gallay19}. This condition can be readily verified for $u$  in \eqref{def:sinym}. This is a main reason while we chose it as prototypical example. In fact, by refining the method of proof, we expect  an analog of Proposition \ref{p:semest} to hold for any shear flow with critical points of order $m$.

In what follows, $C$ denotes a generic constant that may depend on the domain, i.e., on $L_1$ and $L_2$. We utilize standard notation to denote function spaces, e.g. $H^k(\TT^2)$ is the usual $L^2$-based Sobolev space.

Finally, the paper is organized as follows. In Section \ref{s:global}, we obtain the bootstrap estimates and prove Theorem \ref{t:main}.
Then, in Section \ref{s:semigroup}, we establish the exponential stability of the semigroup generated by $H_\nu$ with the improved decay rate, using spectral estimates.

\subsection*{Acknowledgments} The authors thank Tarek Elgindi and Thierry Gallay for insightful discussions. A.M. was partially supported by the US National Science Foundation grants DMS-1909103. M.C.Z. and M.D. acknowledges funding from the Royal Society through a University Research Fellowship (URF\textbackslash R1\textbackslash 191492).

\section{Global existence for the KSE with shear} \label{s:global}

In this section, we establish global existence of solutions of the KSE in the presence of advection by a shear flow with a finite number of critical points.
The semigroup estimate \eqref{bd:semest}   allows to control these growing modes through a suitable decomposition of the solution and a bootstrap argument.

\subsection{Decomposition of the solution and proof of the main result} \label{s:decomposition}

In this section, we derive the system of coupled equations that describe the time evolution of the component $\pz$ of the solution in the kernel of the transport operator and the time evolution of the component  $\pn$ in the orthogonal complement.

 We will refer informally to $\pz$ and $\pn$ as the kernel and projected components, respectively.
Then $\pz$ satisfies
\begin{align}\label{e:pz}
\partial_t\pz+\frac{\nu}{2L_1}\int_{\T^1}|\nabla \pn+\nabla\pz|^2\,\dd x+\nu\partial_{y}^4\pz+\nu\partial_{y}^2\pz=0\,,
\end{align}
while $\pn$ satisfies
\begin{align}\label{e:pn}
\nonumber
\partial_t\pn+u(y)\partial_x\pn+\nu\Delta^2\pn&=-\frac{\nu}{2}|\nabla\pn+\nabla\pz|^2+\frac{\nu}{2L_1}\int_{\T^1}|\nabla\pn+\nabla\pz|^2\,\dd x-\nu\Delta\pn\\
&=-\frac{\nu}{2}|\nabla \pn|^2+\frac{\nu}{2}\lan |\nabla\pn|^2\ran-\nu\partial_y\pn\partial_y\pz-\nu\Delta\pn\,.
\end{align}
We remark that in the equation above the kernel component interacts with the projected ones through the term $\de_y \l \phi\r$. Denoting $\psi=\partial_y\pz$ for notational ease, we have
\begin{align} \label{e:psiEq}
\partial_t\psi + \frac{\nu}{2L_1}\int_{\T^1}\partial_y |\nabla \pn|^2\,\dd x+\nu \psi\partial_y \psi+\nu \partial_y^4\psi +\nu\partial_y^2\psi=0\,.
\end{align}
It was proved in \cite{FM20} that the unique local mild solution to \eqref{eq:AKSE} is also a weak solution satisfying the energy identity on the time of existence of the mild solution. In particular, $\pn\in L^\infty((0,t_0);L^2(\TT^2))\cap L^2((0,t_0);H^2(\TT^2))$, at least for a sufficiently small time $t_0>0$. Furthermore, it was shown in \cite{FM20} that the mild and  weak solution persists as long as its $L^2$ norm is finite, that is, if $T^\ast$ is the maximal time of existence of the solution, then
\[
    T^\ast <\infty \ \Rightarrow \  \limsup_{t\to T^\ast} \|\phi(t)\|_{L^2} =\infty.
\]
Our goal is to obtain a global bound on the $L^2$ norm of the solution via a bootstrap argument, from which global existence follows.
We will employ both energy estimates as well as semigroup estimates to exploit enhanced dissipation arising from the addition of the advection term on $\pn$.  

Let $\cS_t$ be the solution operator from $0$ to time $t\geq 0$ for the transport-hyperdiffusion equation:
\[
     \pa_t g +u(y)\,\pa_x g + \nu \Delta^2 g=0,
\]
that is, $\cS_t=\e^{-t H_\nu}$.
Then $\pn$ satisfies for $0\leq \bar{t}\leq t$,
\begin{align} \label{eq:pnMild}
\pn(t)&=\cS_{t-\bar{t}}(\pn(\bar{t})) + \nonumber\\
& \quad + \int_0^{t-\bar{t}}\cS_{t-\bar{t}-s}\Big(-\frac{\nu}{2}|\nabla\pn(s+\bar{t})|^2
+\frac{\nu}{2}\lan |\nabla\pn(s+\bar{t})|^2\ran -\nu\psi(s)\partial_y\pn(s+\bar{t})-\nu\Delta\pn(s+\bar{t})\Big)\,\dd s,
\end{align}
by Duhamel's principle.
We note that the ``forcing" term under the integral sign on the right-hand side of this equation is well controlled as long as $\pn\in L^\infty((0,t_0);L^2(\TT^2))\cap L^2((0,t_0);H^2(\TT^2))$, provided $\psi$ is also controlled. We stress that the equation for $\pn$ is not autonomous, even though the AKSE is.

Using the decay of $\cS_t$ on the projected component given by \eqref{bd:semest}, it follows from \eqref{eq:pnMild} that, for $0\leq s\leq t$,
\begin{align}\label{e:duhamel}
\nonumber
\norm{\pn(t)}_{L^2}&\leq \norm{\cS_t(\pn(s))}_{L^2}+C\nu\int_0^{t-s}\big(\norm{\nabla \pn}_{L^4}^2+\norm{\psi}_{L^4_y}\norm{\nabla\pn}_{L^4}+\norm{\Delta\pn}_{L^2}\big)(s+\tau)\,\dd \tau\\
\nonumber
&\leq\norm{\cS_t(\pn(s))}_{L^2}+C\nu\int_0^{t-s}\big(\norm{\pn}_{L^2}^{1/2}\norm{\Delta\pn}_{L^2}^{3/2}+\norm{\Delta\pn}_{L^2}\\
\nonumber
&\qquad\qquad\qquad\qquad\qquad\qquad+\norm{\pn}_{L^2}^{1/4}\norm{\Delta\pn}_{L^2}^{3/4}\norm{\psi}_{L^2_y}^{7/8}\norm{\pa^2_y\psi}_{L^2_y}^{1/8}\big)(s+\tau)\,\dd \tau\\
&\leq \norm{\cS_t(\pn(0))}_{L^2}+C\nu\int_0^t\big(\norm{\pn}_{L^2}^{1/2}\norm{\Delta\pn}_{L^2}^{3/2}+\norm{\Delta\pn}_{L^2}+\norm{\pn}_{L^2}^{1/4}\norm{\Delta\pn}_{L^2}^{3/4}\norm{\pa^2_y\psi}_{L^2_y}\big)(s+\tau)\,\dd s\,,
\end{align}
where in the above estimate we used the fact $\norm{\psi}_{L^2_y}\leq C\norm{\pa^2_y\psi}_{L^2_y}$ by applying Poincar\' e's inequality twice (we exploit here that $\psi$ and, hence, all its derivatives have zero average by definition), and the following Gagliardo-Nirenberg interpolation inequalities:
\begin{align}\label{e:soblev1}
\norm{\nabla\pn}_{L^4}\leq C\norm{\pn}_{L^2}^{1/4}\norm{\Delta\pn}_{L^2}^{3/4}\,, \qquad \norm{\psi}_{L^4_y}\leq C\norm{\psi}_{L^2_y}^{7/8}\norm{\pa^2_y\psi}_{L^2_y}^{1/8}  \,.
\end{align}
We next derive some energy estimates that will be needed for the bootstrap argument.
Multiplying \eqref{e:pn} by $\pn$ and integrating by part, using the periodic boundary conditions, yields:
\begin{align} \label{eq:FirstEnergyEst}
\frac{1}{2}\ddt\norm{\pn}_{L^2}^2+\nu \norm{\Delta \pn}_{L^2}^2&=-\frac{\nu}{2}\int_{\T^2}|\nabla \pn|^2\pn\,\dd x\dd y+\frac{\nu}{2L_1}\int_{\T^2}\left(\int_{\T^1}|\nabla\pn|^2\,\dd x\right)\pn\,\dd x\dd y\\
&\qquad \qquad -\nu\int_{\T^2}\psi\partial_y\pn\pn\,\dd x\dd y+\nu\norm{\nabla \pn}_{L^2}^2\\
&\leq C\nu \norm{\nabla\pn}_{L^4}^2\norm{\pn}_{L^2}+C\nu\norm{\psi}_{L^2_y}\norm{\nabla\pn}_{L^4}\norm{\pn}_{L^4}+\nu \norm{\nabla \pn}_{L^2}^2\,.
\end{align}
We recall the Gargliardo-Nirenberg interpolation inequalities in~\eqref{e:soblev1} and
\begin{align}\label{e:soblev}
\norm{\pn}_{L^4}&\leq C\norm{\pn}_{L^2}^{3/4}\norm{\Delta\pn}_{L^2}^{1/4}\,.
\end{align}
These estimate imply:
\begin{align}
\frac{1}{2}\ddt\norm{\pn}_{L^2}^2+\nu \norm{\Delta \pn}_{L^2}^2
&\leq C\nu\norm{\pn}_{L^2}^{3/2}\norm{\Delta\pn}_{L^2}^{3/2}+C\nu\norm{\psi}_{L^2_y}\norm{\pn}_{L^2}\norm{\Delta \pn}_{L^2}\nonumber\\
&\quad+\nu\norm{\pn}_{L^2}\norm{\Delta\pn}_{L^2},
\end{align}
where we have integrated by parts in the last term in \eqref{eq:FirstEnergyEst}.
Applying Young's inequality, we further get
\begin{align}\label{e:energypn}
\ddt\norm{\pn}_{L^2}^2+\nu \norm{\Delta \pn}_{L^2}^2
&\leq C\nu \norm{\pn}_{L^2}^6+C\nu\norm{\pn}_{L^2}^2+C\nu\norm{\pn}_{L^2}^2\norm{\psi}_{L^2_y}^2\,,
\end{align}
using again the 1D Poincar\'e's inequality twice to bound both the $L^2$ norm of $\psi$ as well as  that of $\pa_y \psi$.

We also recall that the enhanced diffusion estimate \eqref{bd:semest} for $S_t$:
\begin{align*}
\norm{\cS_tg}_{L^2}\leq 5\, \e^{-\lambda_\nu t}\norm{g}_{L^2}\,,
\end{align*}
for any $g\in L^2(\T^2)$ with $\int_{\T^1}g(x,y)\,\dd x=0$.
 Above $\lambda_\nu$ satisfies
\begin{align}
\frac{\nu}{\lambda_\nu}\to 0\,,\text{ as $\nu\to 0$.}
\end{align}
In view of \eqref{eq:pnMild}, the regularity of the mild and weak solution and the continuation principle, for all sufficiently small times $t\geq s\geq 0$ we can assume that
\begin{enumerate} [label=(H\arabic*), ref=(H\arabic*)]
\item \label{i:bootstrap1} $\norm{\pn(t)}_{L^2}\leq 8 \e^{-\lambda_\nu (t-s)/4}\norm{\pn(s)}_{L^2}$,
\item  \label{i:bootstrap2} $\nu \int_s^t \norm{\Delta\pn(\tau)}_{L^2}^2\,\dd \tau\leq  4\norm{\pn(s)}_{L^2}^2$.
\end{enumerate}
Let $t_0>0$ be the maximal time such that the estimates above hold on $[0,t_0]$. Following \cite{BS07}, we refer to \mbox{\ref{i:bootstrap1} - \ref{i:bootstrap2}} with $t\in [0,t_0]$ as the {\em bootstrap assumptions}.
The next lemma ensures suitable bounds on $\psi$ once the bootstrap assumptions~\ref{i:bootstrap1} and~\ref{i:bootstrap2} hold.

\begin{lemma}\label{l:linftyl2}
Let  $0<L_2<2\pi$. Assume the bootstrap assumptions~\ref{i:bootstrap1} and~\ref{i:bootstrap2}.  There exists a $\nu$-independent constant $C_1=C_1(\norm{\pn(0)}_{L^2},\norm{\psi(0)}_{L^2_y})$, which can be explicitly
computed, such that 
\begin{align}\label{e:partialpsi}
\norm{\psi(t)}_{L^2_y}^2+\nu\int_0^t\norm{\pa^2_y\psi(s)}_{L^2_y}^2\,\dd s\leq C_1\,,
\end{align}
for all $t\in [0,t_0]$.
\end{lemma}

\begin{proof}
First from the energy estimate and Poincar\'e's inequality, we have
\begin{align}\label{e:psi2}
\frac{1}{2}\ddt\norm{\psi}_{L^2_y}^2+\nu \norm{\pa^2_y\psi}_{L^2_y}^2&=\nu\norm{\partial_y\psi}_{L^2_y}^2+\frac{\nu}{2L_1}\int_{\T^1}|\nabla \pn|^2\partial_y\psi\,\dd x\leq \nu\Big(\frac{L_2}{2\pi}\Big)^2\norm{\partial_y^2\psi}_{L^2_y}^2+\frac{\nu}{2L_1} \norm{\nabla \pn}_{L^4}^2\norm{\partial_y\psi}_{L^2_y}\,.
\end{align}
It follows from the Gagliardo-Nirenberg inequality  that
\begin{align}
&\norm{\partial_y\psi}_{L^2_y}^2\leq \norm{\pa^2_y\psi}_{L^2_y}\norm{\psi}_{L_y^2},\\
&\norm{\nabla \pn}_{L^4}\leq C\norm{\pn}_{L^2}^{1/4}\norm{\Delta\pn}_{L^2}^{3/4}.
\end{align}
Appealing to the two bounds above, estimate~\eqref{e:psi2} becomes
\begin{align*}
\frac{1}{2}\ddt\norm{\psi}_{L^2_y}^2+\nu\Big(1-\big(\frac{L_2}{2\pi}\big)^2\Big)\norm{\pa^2_y\psi}_{L^2_y}^2
&\leq C\nu\norm{\pn}_{L^2}^{1/2}\norm{\Delta \pn}_{L^2}^{3/2}\norm{\de_{y}^2\psi}_{L_y^2}^{1/2}\norm{\psi}_{L_y^2}^{1/2}\,.
\end{align*}
 It then follows by Young's inequality that
\begin{align}
\label{e:psitmp1}
\ddt\norm{\psi}_{L^2_y}^2+\nu\Big(1-\big(\frac{L_2}{2\pi}\big)^2\Big) \norm{\pa^2_y\psi}_{L^2_y}^2&\leq  C\nu \norm{\pn}_{L^2}^{2/3}\norm{\Delta\pn}_{L^2}^2\norm{\psi}_{L^2_y}^{2/3} \\
\label{e:psitmp2}
&\leq C\nu \norm{\pn}_{L^2}^{2/3}\norm{\Delta\pn}_{L^2}^2+C\nu \norm{\pn}_{L^2}^{2/3}\norm{\Delta\pn}_{L^2}^2\norm{\psi}_{L^2_y}^2\,.
\end{align}
We define an integrating factor $\mu=\exp\big(-C\nu\int_0^t\norm{\pn}_{L^2}^{2/3}\norm{\Delta\pn}_{L^2}^2\,\dd s\big)$. Then solving~\eqref{e:psitmp2} gives
\begin{align} \label{eq:psiEst2}
\norm{\psi(t)}_{L^2_y}^2&\leq C\nu \mu^{-1}\int_0^t\norm{\pn(s)}_{L^2}^{2/3}\norm{\Delta\pn(s)}_{L^2}^2\,\dd s+\mu^{-1}\norm{\psi(0)}_{L^2_y}^2\nonumber\\
&\leq \left(16 C\e^{16C\norm{\pn(0)}_{L^2}^{8/3}}\norm{\pn(0)}_{L^2}^{8/3}+\e^{16C\norm{\pn(0)}_{L^2}^{8/3}}\norm{\psi(0)}_{L^2_y}^2\right):=C_1\,,
\end{align}
where the last inequality follows by the bootstrap assumptions~\ref{i:bootstrap1} and~\ref{i:bootstrap2}. By using~\eqref{eq:psiEst2} in \eqref{e:psitmp1}, we get~\eqref{e:partialpsi}.
\end{proof}

We show below in Subsection \ref{s:bootstrap} that, in fact,  there exists $\nu_0>0$ small enough such that, if $\nu<\nu_0$, then for any $0\leq s\leq t\leq t_0$,
\begin{enumerate} [label=(B\arabic*), ref=(B\arabic*)]
\item \label{i:bootstrap3} $\norm{\pn(t)}_{L^2}\leq 4 \e^{-\lambda_\nu (t-s)/4}\norm{\pn(s)}_{L^2}$,
\item  \label{i:bootstrap4} $\nu \int_s^t \norm{\Delta\pn(\tau)}_{L^2}^2\,\dd \tau\leq  2\norm{\pn(s)}_{L^2}^2$\,.
\end{enumerate}
 We refer to \ref{i:bootstrap3}-\ref{i:bootstrap4} as the {\em bootstrap estimates}.
Assuming temporarily this fact, we proceed with 
the proof of  Theorem \ref{t:main}.

\begin{proof}[Proof of Theorem \ref{t:main}]
From Lemma~\ref{lem:B2} and Lemma~\ref{lem:B1} and the definition of $t_0$, we must have $t_0=\infty$. In particular, $\pn\in L^\infty([0,\infty);L^2(\TT^2))\cap L^2([0,\infty);H^2(\T^2))$. By Lemma~\ref{l:linftyl2}, we have $\psi\in L_{loc}^\infty([0,\infty);L^2(\T^1))\cap L_{loc}^2([0,\infty);H^2(\T^1))$. By the triangle and Poincar\'e's inequalities, it then follows $\Delta\phi\in L_{loc}^2([0,\infty);L^2(\T^2))$. If we further denote
\begin{align}\label{eq:defphibar}
\bar \phi=\frac{1}{L_1L_2}\int_{\T^2}\phi(x,y)\,\dd x\dd y= \frac{1}{L_2}\int_{\T^1}\pz\,\dd y,
\end{align}
then from \eqref{e:pz} we have
\begin{align}
\partial_t\bar \phi&=-\frac{\nu}{2L_1L_2}\int_{\T^2}|\nabla \pn+\nabla\pz|^2\,\dd x\dd y =-\frac{\nu}{2L_1L_2}\int_{\T^2} |\nabla \pn|^2\,\dd x\dd y-\frac{\nu}{2L_2}\int_{\T^1}|\psi|^2\,\dd y\,.
\end{align}
By integrating the above equation and applying estimate~\eqref{e:partialpsi} and~\ref{i:bootstrap4}, we obtain $\bar\phi\in L_{loc}^\infty ([0,\infty))$. By Lemma~\ref{l:linftyl2}, it follows that
$\psi\in L_{loc}^\infty([0,\infty);L^2(\T^1))$. Hence, \eqref{eq:defphibar} and the Poincar\'e inequality imply
that $\pz\in L_{loc}^\infty([0,\infty);L^2(\T^1))$. As a consequence, $\phi\in L_{loc}^\infty([0,\infty);L^2(\T^2))$.
Finally, we note that $\nabla^2 \phi=\nabla^2 \pn +\nabla \psi$ so that $\phi \in L^2_{loc}([0,\infty);H^2(\T^2))$.
This concludes the proof.
\end{proof}

\subsection{Bootstrap estimates} \label{s:bootstrap}

It remains to establish the bootstrap estimates \ref{i:bootstrap3}-\ref{i:bootstrap4}, which we accomplish through a series of lemmas. 
We address \ref{i:bootstrap4} first. 

\begin{lemma}\label{lem:B2}
Let  $0<L_2<2\pi$. Assume the bootstrap assumptions~\ref{i:bootstrap1} and~\ref{i:bootstrap2}. There exists $\nu_0=\nu_0(\norm{\pn(0)}_{L^2})$, explicitly computable, with the following property: for any $0\leq s\leq t \leq t_0$ and for any $\nu\leq \nu_0$,
 it holds that 
\begin{align}
\nu \int_s^t \norm{\Delta\pn(\tau)}_{L^2}^2\,\dd \tau\leq  2\norm{\pn(s)}_{L^2}^2\,.
\end{align}
In particular, \ref{i:bootstrap4} holds.
\end{lemma}

\begin{proof}
 The bootstrap assumptions, Lemma~\ref{l:linftyl2} and the energy estimate~\eqref{e:energypn} give
\begin{align}
\nu\int_s^t\norm{\Delta\pn(\tau)}_{L^2}^2\,\dd \tau&\leq \norm{\pn(s)}_{L^2}^2+C\nu\int_s^t \norm{\pn(\tau)}_{L^2}^6+(1+C_1)\norm{\pn(\tau)}_{L^2}^2\,\dd \tau\notag\\
&\leq \norm{\pn(s)}_{L^2}^2+\frac{\nu}{\lambda_\nu}C(\norm{\pn(0)}_{L^2}^4+1+C_1)\norm{\pn(s)}_{L^2}^2
\label{bd:L2L2Delta}
\end{align}
We observe that, since $\nu/\lambda_\nu\to 0$ as $\nu\to 0$, there exists $\nu_0$ such that
$$
    \frac{\nu_0}{\lambda_{\nu_0}} \leq \frac{1}{C(\norm{\pn(0)}_{L^2}^4+1+C_1)} .
$$
Hence, by combining the choice $\nu_0$ with \eqref{bd:L2L2Delta} we conclude the proof of the Lemma.
\end{proof}

It remains to prove \ref{i:bootstrap3}, which we accomplish in different steps. The next lemma states that within a fixed length of time, the quantity  $\norm{\pn}_{L^2}$ will never grow too fast.

\begin{lemma}\label{l:slowgrowth}
Let  $0<L_2<2\pi$. Assume the bootstrap assumptions~\ref{i:bootstrap1} and~\ref{i:bootstrap2}, and fix $\tau^*=4/\lambda_\nu$.
 For any  $0\leq t_1<t_0$, there exists $\nu_0=\nu_0(\norm{\pn(0)}_{L^2})$ such that  for any $\nu\leq \nu_0$ 
 there holds
\begin{align}
\norm{\pn(t)}_{L^2}\leq \sqrt 2\norm{\pn(t_1)}_{L^2}\,.
\end{align}
for  all  $t\in [t_1, t_1+\tau^*]\cap [0,t_0]$.  
\end{lemma}

\begin{proof}
We will assume that $\nu_0$ is small enough so that Lemma \ref{lem:B2} applies.
Again by Lemma~\ref{l:linftyl2}, the energy estimate~\eqref{e:energypn}, and the bootstrap assumption \ref{i:bootstrap1},  for some positive
$C_2=C_2(\norm{\pn(0)}_{L^2},\norm{\psi(0)}_{L^2_y})$ we have that 
\begin{align}
\ddt\norm{\pn}_{L^2}^2&\leq C\nu\norm{\pn}_{L^2}^6+C (1+C_1) \nu \norm{\pn}_{L^2}^2\notag\\
&\leq C(1+C_1)\nu \norm{\phi_{\neq}(0)}_{L^2}^2(\norm{\phi_{\neq}}_{L^2}^4+1)\notag\\
&\leq  C_2\nu \left(\norm{\phi_{\neq}}_{L^2}^4+1\right).\label{e:stima}
\end{align}
Now, we define $T(B)$ as
\begin{align}
 T(B)=\frac{1}{\nu C_2}\int_{B^2}^{2 B^2}\frac{\dd y}{y^2+1}=\frac{1}{\nu C_2}(\arctan(2B^2)-\arctan(B^2)).
\end{align}
 It is easy to see that  $T(\cdot)$ is a decreasing function and, since by the bootstrap assumption   \ref{i:bootstrap1} we have that $\norm{\pn(t_1)}_{L^2}\leq 8 \norm{\pn(0)}_{L^2}$,
 it follows that $T(\norm{\pn(t_1)}_{L^2})\geq T(8\norm{\pn(0)}_{L^2})$.
In light of \eqref{e:stima}, by the definition of $T(B)$ we have that for any $t\in [t_1, t_1+T(8\norm{\pn(0)}_{L^2})]\cap [0,t_0]$,  the estimate $\norm{\pn(t)}_{L^2}\leq \sqrt 2\norm{\pn(t_1)}_{L^2}$ holds. The lemma is now proved if we choose $\tau^*\leq T(8\norm{\pn(0)}_{L^2})$, which is equivalent to asking for
\begin{align}\label{e:nu021}
\frac{\nu}{\lambda_{\nu}}\leq \frac{1}{4 C_2} (\arctan(128 \norm{\pn(0)}_{L^2}^2)-\arctan(64 \norm{\pn(0)}_{L^2}^2)).
\end{align}
Finally, $\nu/\lambda_\nu\to 0$ as $\nu\to0$, so it is enough to satisfy \eqref{e:nu021} for $\nu_0$. This concludes the proof.
\end{proof}

The next lemma shows that by, selecting $\nu$ sufficiently small, a constant fraction  of $\norm{\pn}_{L^2}$ decays after a fixed length of time.
\begin{lemma}\label{l:decaytau}
Let  $0<L_2<2\pi$. Assume the bootstrap assumptions~\ref{i:bootstrap1} and~\ref{i:bootstrap2}, and fix again  $\tau^*=4/\lambda_\nu$. If $t_0\geq \tau^*$ then there exists $\nu_0=\nu_0(\norm{\pn(0)}_{L^2})$, explicitly computable, with the following property: for any $s\in [0,t_0-\tau^*]$ and for any $\nu\leq \nu_0$,
 \begin{align}
 \norm{\pn(\tau^*+s)}_{L^2}\leq \frac{1}{\e}\norm{\pn(s)}_{L^2}\,.
 \end{align}
\end{lemma}

\begin{proof}
In the course of this proof, we assume that $\nu_0$ is small enough so that Lemma \ref{lem:B2} can be applied.
By the definition of $\tau^*$, we have
\begin{align}
\norm{\cS_{\tau^*}(\pn(s))}_{L^2}\leq \frac{5}{\e^4}\norm{\pn(s)}_{L^2}\leq \frac{1}{\e^2}\norm{\pn(s)}_{L^2}\,.
\end{align}
Using this inequality in~\eqref{e:duhamel} yields 
\begin{align}
\norm{\pn(\tau^*+s)}_{L^2}&\leq \frac{1}{\e^2}\norm{\pn(s)}_{L^2}+C\nu\int_s^{\tau^*+s}\big(\norm{\pn}_{L^2}^{1/2}\norm{\Delta\pn}_{L^2}^{3/2}+\norm{\Delta\pn}_{L^2}+\norm{\pn}_{L^2}^{1/4}\norm{\Delta\pn}_{L^2}^{3/4}\norm{\pa^2_y\psi}_{L^2_y}\big)\,\dd \tau \nonumber\\
&\leq  \frac{1}{\e^2}\norm{\pn(s)}_{L^2}+C\Big(\nu\int_s^{\tau^*+s}\norm{\pn(\tau)}_{L^2}^2\,\dd \tau\Big)^{1/4}\Big(\nu\int_s^{\tau^*+s}\norm{\Delta\pn(\tau)}_{L^2}^2\,\dd \tau\Big)^{3/4}\nonumber\\
&\qquad +C\Big(\nu\int_s^{\tau^*+s}\norm{\Delta\pn(\tau)}_{L^2}^2\,\dd \tau\Big)^{1/2}(\nu\tau^*)^{1/2}\nonumber\\
&\qquad+C\Big(\nu\int_s^{\tau^*+s}\norm{\Delta\pn(\tau)}_{L^2}^2\,\dd \tau\Big)^{3/8}
\Big(\nu\int_s^{\tau^*+s}\norm{\pa^2_y\psi(\tau)}_{L^2}^2\,\dd \tau\Big)^{1/2}\Big(\nu\int_s^{\tau^*+s}\norm{\pn(\tau)}_{L^2}^2\,\dd \tau\Big)^{1/8}\,.
\end{align}
Using the bootstrap assumptions \ref{i:bootstrap1}-\ref{i:bootstrap2}   and Lemma~\ref{l:linftyl2}, it then follows that
\begin{align}\label{e:ldecaytmp}
\norm{\pn(\tau^*+s)}_{L^2}&\leq \frac{1}{\e^2}\norm{\pn(s)}_{L^2}+C(\nu\tau^*)^{1/4}\norm{\pn(s)}_{L^2}^2+C(\nu\tau^*)^{1/2}\norm{\pn(s)}_{L^2}+C\sqrt{C_1}(\nu\tau^*)^{1/8}\norm{\pn(s)}_{L^2}\nonumber\\
&\leq \frac{1}{\e^2}\norm{\pn(s)}_{L^2}+C(\nu\tau^*)^{1/8}\big(\norm{\pn(s)}_{L^2}+\sqrt{C_1}\big)\norm{\pn(s)}_{L^2}\,,
\end{align}
where we used the fact that $\nu\tau^*\ll 1$ when $\nu_0 $ is small enough. By further restring $\nu_0$ so that
 \begin{align}\label{e:nu03}
 \frac{1}{\e^2} +C(4\nu_0\lambda_{\nu_0}^{-1})^{1/8}\big(8\norm{\pn(0)}_{L^2}+\sqrt{C_1}\big)\leq \frac{1}{\e},
  \end{align}
 the desired result follows from~\eqref{e:ldecaytmp}.
\end{proof}

Now we are ready to show that the bootstrap assumption~\ref{i:bootstrap1} can be refined.
\begin{lemma}\label{lem:B1}
Let $0<L_2<2\pi$. Assume the bootstrap assumptions~\ref{i:bootstrap1} and~\ref{i:bootstrap2}. There exists $\nu_0=\nu_0(\norm{\pn(0)}_{L^2})$, explicitly computable, with the following property: for any $0\leq s\leq t\leq t_0$ and for any $\nu\leq \nu_0$,
 it holds that 
\begin{align}
\norm{\pn(t)}_{L^2}\leq 4\e^{-\lambda_\nu (t-s)/4}\norm{\pn(s)}_{L^2}\,.
\end{align}
In particular, \ref{i:bootstrap3} holds.
\end{lemma}

\begin{proof}
We fix $\nu_0 =\nu_0(\norm{\pn(0)}_{L^2})$ so that all the restrictions in Lemmata  \ref{lem:B2}-\ref{l:decaytau} are fulfilled. If $t_0<\tau^*$ then \ref{i:bootstrap3} directly follows by Lemma \ref{l:slowgrowth} since $\sqrt{2}\e<4$. When $t_0\geq \tau^*$, by Lemma~\ref{l:decaytau}, we have
\begin{align}
\norm{\pn(n\tau^*+s)}_{L^2}\leq \e^{-n}\norm{\pn(s)}_{L^2}\,,\quad\text{for any $n\in \ZZ_+$ satisfying $s+n\tau^*\leq t_0$.}
\end{align}
For any $0\leq s\leq t\leq t_0$, there exists $n$ such that $t\in[n\tau^*+s,~(n+1)\tau^*+s)$. From Lemma~\ref{l:slowgrowth} with $t_1=n\tau^*+s$, it follows that
\begin{align}
\norm{\pn(t)}_{L^2}\leq \sqrt 2 \norm{\pn(n\tau^*+s)}_{L^2}\leq \sqrt 2 \e^{-n}\norm{\pn(s)}_{L^2}\leq \sqrt 2 \e^{1-(t-s)/\tau^*}\norm{\pn(s)}_{L^2}\leq 4\e^{-\lambda_\nu (t-s)/4}\norm{\pn(s)}_{L^2}\,.
\end{align}
This concludes the proof of the lemma.
\end{proof}

\section{Semigroup estimates} \label{s:semigroup}

In this section, we prove Proposition \ref{p:semest}, namely an improved decay estimate for the semigroup generated by $H_\nu$ in $L^2$, under a general condition on the shear velocity profile $u$.

We denote by $\rL^2(\TT^2)$ the closed subspace of $L^2(\TT^2)$ of elements for which $\gz=0$.
By Fubini-Tonelli's Theorem, such elements are also mean-zero on the torus. We will be concerned with the restriction of the operator  $H_\nu$ to $\rL^2(\TT^2)$ viewed as an unbounded operator.  By slight abuse of notation, we denote the restriction also by $H_\nu$. It is straightforward to check that the projection onto
$\rL^2(\TT^2)$ commutes with the semigroup $\e^{-tH_\nu}$ generated by $H_\nu$.

Let $(X,\|\cdot\|)$ be a complex Hilbert space and let $H$ be a closed, densely defined operator on $X$.
As shown in \cite{Wei18}, if $H$ is an $m$-accretive operator on $X$, then the decay properties of the semigroup $\e^{-t H}$ can be understood in terms of the following quantity:
\begin{equation}
\Psi(H)=\inf\left\{\|(H-i\lambda)g\|: g\in D(H),\,\lambda\in \R,\, \|g\|=1\right\},
\end{equation}
is related to the \textit{pseudospectral} properties of the operator \cite{GGN09}.
Following \cite{Wei18}, for $\frac{L_1k}{2\pi}\in  \ZZ_*$ and $\nu\in(0,1]$, we consider the operator $H_\nu$ localized to the $k$th Fourier mode in the direction of the shear, namely, the operator
\begin{equation}
H_{\nu,k}=\nu \Delta_k^2 + iku(y),\qquad \Delta_k:=-k^2+\de_{yy}.
\end{equation}
Following the arguments in \cite{Wei18} for the Laplace operator, it can be shown that $H_{nu,k}$ an $m$-accretive operator on $L^2(\TT^1)$ with domain $H^4(\TT^1)$. Here, $L^2$ is a space of complex-valued functions.
Then, as a consequence of \cite{Wei18}*{Theorem 1.3},
\begin{equation}
	\label{bd:Wei}
\|\e^{-H_{\nu,k} t}\|_\op \leq \e^{-t\Psi(H_{\nu,k})+\pi/2}, \qquad \forall t\geq 0,
\end{equation}
where $\|\cdot\|_\op$ denotes the operator norm.
To establish lower bounds on $\Psi(H_{\nu,k})$, we  assume the following condition on the shear flow.

\begin{assumption}\label{a:lowerbddE}
There exist $m, N\in \N$, $c_1>0$ and  $\delta_0\in(0,L_2)$ with the property that, for
  any $\lambda \in \R$ and any $\delta\in(0,\delta_0)$, there exist $n\leq N$ and points $y_1,\ldots y_n\in [0,L_2)$
such that
\begin{align}\label{eq:lowerbddE}
|u(y)-\lambda|\geq c_1 \left(\frac{\delta}{L_2}\right)^m, \qquad \forall \  |y-y_j|\geq \delta, \quad \forall j\in \{1,\ldots n\}.
\end{align}
\end{assumption}


\begin{remark}
Assumption \ref{a:lowerbddE} is heavily inspired by a similar property of the velocity field associated to the Oseen's vortex \cite{LWZ20}. In \cite{gallay19}, Gallay previously observed that the method of proof in \cite{LWZ20} can be extended to more general shear flows assuming a condition similar to \eqref{eq:lowerbddE}.
\end{remark}

The following  is the main result of this section.

\begin{proposition} \label{t:SemigroupEst}
Let $u$ satisfy Assumption \ref{a:lowerbddE}.  Assume $k\ne 0$ and $\nu|k|^{-1}\leq 1$. There exists a constant $\eps'_0>0$, independent of $\nu$ and $k$, such that
\begin{align}
	\label{bd:pseudo}
\Psi (H_{\nu,k})\geq \eps'_0 \nu^\frac{m}{m+4}|k|^\frac{4}{m+4}.
\end{align}
\end{proposition}

We state next a direct consequence of the theorem.

\begin{corollary} \label{c:SemigroupEst}
In the hypotheses of Proposition \ref{t:SemigroupEst}, let $P_k$ denotes the $L^2$ projection onto the $k$-th Fourier mode in the horizontal direction. Then, for every $t\geq 0$,
\begin{align} \label{eq:DecayRate}
\|\e^{-H_\nu t} \, P_k\|_\op \leq \e^{- \eps'_0 \nu^\frac{m}{m+4}|k|^\frac{4}{m+4} t +\pi/2}.
\end{align}
In particular, $H_\nu$ generates an exponentially stable semigroup in $\rL^2(\TT^2)$ with rate:
\begin{equation}\label{bd:sembound}
   \|\e^{-H_\nu t}\|_\op \leq  \e^{-\lambda'_\nu t +\pi/2},  \qquad t>0,
\end{equation}
where     $\lambda'_\nu=\eps'_0 \nu^\frac{m}{m+4}$ for some $\eps_0'>0$.
\end{corollary}

Before proving Proposition \ref{t:SemigroupEst}, we show that the Assumption \ref{a:lowerbddE} is satisfied for  $u$ as in \eqref{def:sinym}.

\begin{example} 
We consider the case of $u(y)=(\sin(y))^m$, defined on  $[0,2\pi)$. For a general period $L_2$, the result follows by a standard rescaling argument. Without loss of generality, we may assume that  $\delta_0>0$ is small enough
so that $\cos (\delta_0)\geq 1/2$. In particular, for every $C\geq 1$,
\begin{align} \label{eq:delta0choice}
\sin (\delta/C)\geq \frac{\delta}{2C}, \qquad \forall \delta\in (0,\delta_0).
\end{align}
Given $\lambda \in \mathbb{R}$, we choose the set of points $y_1,\dots,y_n$ to be the union of the set of the critical points of $u$ with the set $u^{-1}(\{\lambda\})$ (possibly empty). More precisely, we consider
\begin{align}
\label{def:Y}
Y:=\begin{cases}\{\pi/2,3\pi/2\}\cup u^{-1}(\{\lambda\}), \quad &m=1,\\  
\{0,\pi/2,\pi,3\pi/2\}\cup u^{-1}(\{\lambda\}), \quad &m\geq 2,
\end{cases} \qquad n:=|Y|.
\end{align}
There are at most $N=8$ points in $Y$. Observe that the critical points at $y=\pi/2, 3\pi/2$ are such that $u''(\pi/2),u''(3\pi/2)\neq0$ for any $m\geq 1$. On the other hand, for $m\geq 2$ the critical points  at $y=0,\pi$ are such that $u^{(j)}(0)=u^{(j)}(\pi)=0$ for every $j=1,\ldots,m-1$ and $u^{(m)}(0),u^{(m)}(\pi)\neq 0$. We denote $Y=\{y_{i}\}_{i=1}^n$ and we order the points in such a way that $0\leq y_i\leq y_{i+1}\leq 2\pi$ for $i=1,\ldots,N$. We fix $\delta\in (0,\delta_0)$ and consider different cases.

\medskip

\noindent {\bf Case  $\lambda\in[0,1]$.} 
Denote by $y_\lambda$ the smallest element of $u^{-1}(\{\lambda\})$. Due to the symmetries of $u$, we know that $y_\lambda \in[0,\pi/2]$. We consider three situations.

\begin{itemize}[leftmargin=*]
 \item $|y_\lambda|<\delta/4$: In this case, thanks to \eqref{eq:delta0choice}, we have
\begin{align}
|u(y)-\lambda|&=|(\sin (y))^m-(\sin(y_\lambda))^m|
\geq  |\sin (\delta)|^m-|\sin(y_\lambda)|^m\notag\\
&\geq \frac{\delta^m}{2^m} - |\sin(\delta/4)|^m\geq \frac{\delta^m}{2^m} -\frac{\delta^m}{4^m}\geq  \frac{\delta^m}{2^{m+1}},
\end{align}
where we used that $|\sin(y)|\leq |y|$.

\item $|y_\lambda-\pi/2|<\delta/4$: In this case we have
\begin{align}
|u(y)-\lambda|&=|(\sin (y))^m-(\sin(y_\lambda))^m|\geq |(\sin (\pi/2-\delta))^m-(\sin(y_\lambda))^m|\notag\\
&\geq  |(\sin (\pi/2-\delta)))^m-(\sin(\pi/2)))^m|- |(\sin (\pi/2)))^m-(\sin(\pi/2-\delta/4)))^m|.
\end{align}
Since $\de_y (\sin(y))^m= m(\sin(y))^{m-1}\cos(y)$ and
$\de^2_y (\sin(y))^m= m(m-1)(\sin(y))^{m-2}\cos^2 (y)-m(\sin(y))^m$,
\begin{align}
 |(\sin (\pi/2-\delta)))^m-(\sin(\pi/2)))^m|=\left|\left(1-\frac{\delta^2}{2}+O(\delta^4)\right)^m-1\right|=\frac{m}{2}\delta^2+O(\delta^4).
\end{align}
so that, by possibly restricting the size of  $\delta_0$ further (depending on $m$),  we also have
\begin{align}
 |(\sin (\pi/2-\delta)))^m-(\sin(\pi/2)))^m|\geq \frac{m}{4}\delta^2.
\end{align}
On the other hand,
\begin{equation}
 |(\sin (\pi/2)))^m-(\sin(\pi/2-\delta/4)))^m|\leq \frac{m}{8}\delta^2.
\end{equation}
Consequently,
\begin{align}
|u(y)-\lambda|\geq \frac{m}{8}\delta^2.
\end{align}

\smallskip

\item  $|y_\lambda|\geq \delta/4$ and $|y_\lambda-\pi/2|\geq\delta/4$: In this case,
we have
\begin{align}
|u(y)-\lambda|&=|(\sin (y))^m-(\sin(y_\lambda))^m|\notag\\
&\geq \min\left\{ |(\sin (y_\lambda+\delta/8))^m-(\sin(y_\lambda))^m|,|(\sin (y_\lambda-\delta/8))^m-(\sin(y_\lambda))^m|\right\}.
\end{align}
We observe that $y_\lambda+\delta/8<\pi/2-\delta/8$ and $y_\lambda-\delta/8>\delta/8$. Then if $y_\lambda\geq \pi/4$,
it follows by the Mean Value Theorem that  for some $\xi\in (y_\lambda,y_\lambda+\delta/8)$,
\begin{align}
|(\sin (y_\lambda+\delta/8))^m-(\sin(y_\lambda))^m|&=m|(\sin(\xi))^{m-1}\cos(\xi)|\frac{\delta}{8}\notag\\
&\geq m|(\sin(\pi/4))^{m-1}\cos(\pi/2-\delta/8)|\frac{\delta}{8}\notag\\
&\geq \frac{m}{2^{\frac{m-1}{2}}}|\sin(\delta/8)|\frac{\delta}{8}\geq\frac{m}{2^{\frac{m-1}{2}}}\frac{\delta^2}{128},
\end{align}
also using  \eqref{eq:delta0choice}. 
Otherwise, if  $y_\lambda\leq \pi/4$, we have
\begin{align}
|(\sin(y_\lambda+\delta/8))^m-(\sin(y_\lambda))^m|&=m|(\sin(\xi))^{m-1}\cos(\xi)|\frac{\delta}{8}\notag\\
&\geq m|(\sin(\delta/4))^{m-1}\cos(\pi/4+\delta/8)|\frac{\delta}{8}\notag\\
&\geq m|(\sin(\delta/4))^{m-1}\cos(\pi/3)|\frac{\delta}{8}\geq \frac{m}{8^{m+1}}\delta^{m}.
\end{align}
A lower bound for $|(\sin (y_\lambda-\delta/8)))^m-(\sin(y_\lambda)))^m|$ can be proved in a similar way. 
\end{itemize}

\medskip

\noindent{\bf Case  $\lambda>1$.} In this case $u^{-1}(\{\lambda\})=\emptyset$. Hence, since ${\rm dist}(y,Y)\geq \delta$, we simply have
\begin{align}
|u(y)-\lambda|&\geq|(\sin(y))^m-(\sin(\pi/2))^m|\geq  |(\sin (\pi/2-\delta))^m-(\sin(\pi/2))^m|\notag\\
&\geq m|(\sin(\xi))^{m-1}\cos(\xi)| \delta\geq m|(\sin(\pi/6))^{m-1}\cos(\pi/2-\delta)| \delta\geq \frac{m}{2^{m}} \delta^2,
\end{align}
where we used \eqref{eq:delta0choice} in the last inequality.

\medskip

\noindent{\bf Case  $\lambda<0$.}
This case is only relevant for $m$ even. In fact,  for $m$ odd we proceed as for the case $\lambda\in[0,1]$ by symmetry.
Here, again $u^{-1}(\{\lambda\})=\emptyset$. Hence, since ${\rm dist}(y,Y)\geq \delta$, we have
\begin{align}
|u(y)-\lambda|&\geq|(\sin (y))^m|\geq  (\sin (\delta))^m\geq   \frac{\delta^m}{2^{m}},
\end{align}
by \eqref{eq:delta0choice}. 

Hence, there exists a constant $c_m>0$ such that
\begin{align}
|u(y)-\lambda|\geq c_m \delta^{\max\{m,2\}},
\end{align}
as we wanted.

\end{example}

We now turn our attention to the proof of Proposition \ref{t:SemigroupEst}.

\begin{proof}[Proof of Proposition \ref{t:SemigroupEst}]
The theorem follows by establishing a lower bound on $\Psi(H_{\nu,k})$. In the following, $\|\cdot\|$ denotes the $L^2$ norm and $\langle \cdot,\cdot \rangle$ denotes the Hermitian inner
product in $L^2$.

We fix $\lambda\in\R$ and pick  $g\in D(H_{\nu,k})$ with $\|g\|=1$. For notational convenience, we set
\[
   H:=H_{\nu,k}-i \lambda=\nu \Delta_k^2+i k(u(y)-\lambda).
\]
 Let $\chi:[0,L_2)\to[-1,1]$ be a smooth approximation of $\sign(u(y)-\lambda)$  such that $\|\chi'\|_{L^\infty}\leq c_2\delta^{-1}$,
$\|\chi''\|_{L^\infty}\leq c_2\delta^{-2}$, $\chi (u-\lambda)\geq 0$ and
\begin{align}
\chi(y)(u(y)-\lambda)=|u(y)-\lambda|, \quad \text{whenever} \quad |y-y_j|\geq \delta,\quad \forall j\in \{1,\ldots n\},
\end{align}
where $y_j$ are the points in Assumption \ref{a:lowerbddE}. The function $\chi$ can be constructed via a standard mollification argument. 
A double integration by parts in $y$ yields the identity
\begin{align}
\Re\l H g,g\r=\nu \|\Delta_k g\|^2,
\end{align}
which implies that
\begin{align}\label{eq:repartcons}
 \|\Delta_kg\|^2\leq \frac{1}{\nu}\| H g\|\|g\|.
\end{align}
On the other hand, we have
\begin{align}
\l H g, \chi g\r&= \nu\l \Delta_k^2  g, \chi g\r +ik \l   (u(y)-\lambda) g, \chi g\r \notag\\
&= \nu\l \Delta_k g, \chi'' g\r +2\nu\l  \Delta_k g, \chi' \de_y g\r +\nu\l  \Delta_k g, \chi \Delta_k  g\r +ik \l   (u(y)-\lambda) g, \chi g\r
\end{align}
so that
\begin{align}
\Im\l H g,\chi g\r=\nu \Im\l \Delta_k  g, \chi'' g\r +2\nu \Im \l\Delta_k g, \chi'\de_yg\r +k \l   (u(y)-\lambda) g, \chi g\r.
\end{align}
In particular, from the properties of the function $\chi$ and the interpolation inequality $ \|\de_yg\|^2\leq\|\Delta_k  g \|\|g\|$, it follows that
\begin{align}\label{eq:impartcons}
|k| \l   (u(y)-\lambda) g, \chi g\r\leq \| Hg\| \|g\|+\frac{c_2\nu}{\delta^2} \|\Delta_k  g \|\|g\|+\frac{c_2\nu}{\delta}  \|\Delta_k  g \|^{3/2}\|g\|^{1/2}.
\end{align}
 We denote
\begin{align}
  E:=\left\{ y\in [0,L_2): |y-y_j|\geq \delta, \quad \text{for } j=1,\ldots,n\right\},
\end{align}
where $y_j$ are the points in Assumption \ref{a:lowerbddE}. By \eqref{eq:lowerbddE} we have
\begin{align}
 \l   (u(y)-\lambda) g, \chi g\r\geq \int_E |u(y)-\lambda| |g(y)|^2\dd y\geq c_1 \left(\frac{\delta}{L_2}\right)^m  \int_E |g(y)|^2\dd y.
\end{align}
Utilizing \eqref{eq:repartcons} and \eqref{eq:impartcons}, we find that there exists a positive constant $\tilde{c}_2$ such that
\begin{align}\label{eq:estimateE}
&\int_E |g(y)|^2\dd y\leq \frac{1}{c_1 |k|} \left(\frac{L_2}{\delta}\right)^{m}\left[\| Hg\| \|g\|+\frac{c_2\nu}{\delta^2} \|\Delta_k  g \|\|g\|+\frac{c_2\nu}{\delta}  \|\Delta_k  g \|^{3/2}\|g\|^{1/2}\right]\notag\\
&\qquad\leq \frac{1}{c_1 |k|} \left(\frac{L_2}{\delta}\right)^{m}\| Hg\| \|g\|
+\tilde{c}_2\left(\left(\frac{\nu}{|k|\delta^2}\right)^2\left(\frac{L_2}{\delta}\right)^{2m} +\left(\frac{\nu}{|k|\delta}\right)^{4/3}\left(\frac{L_2}{\delta}\right)^{\frac{4m}{3}} \right)\|\Delta_k  g \|^2 +\frac14\|g\|^2\notag\\
&\qquad\leq \left(\frac{1}{c_1 |k|}\left(\frac{L_2}{\delta}\right)^{m}+\frac{\tilde{c}_2}{\nu}\left(\left(\frac{\nu}{|k|L_2^2}\right)^2\left(\frac{L_2}{\delta}\right)^{2m+4} +\left(\frac{\nu}{|k|L_2}\right)^{4/3}\left(\frac{L_2}{\delta}\right)^{\frac{4}{3}(m+1)} \right)\right)\| Hg\| \|g\|+\frac14\|g\|^2.
\end{align}
On the other hand, since $E^c$ is of size at most $N\delta$, we have
\begin{align}\label{eq:estimateEc}
\int_{E^c} |g(y)|^2\dd y\leq N\delta \|g\|^2_{L^\infty}&\leq N\delta \left(2\|g\|\|\de_yg\| +\frac{1}{L_2} \| g\|^2\right)\notag\\
&\leq N\delta \left(2\|g\|^{3/2}\| \Delta_k g\|^{1/2} +\frac{1}{L_2} \| g\|^2\right)\notag\\
&\leq  \frac{(6 N\delta)^4}{12}\| \Delta_k g\|^2 +\left(\frac{N\delta}{L_2} +\frac{1}{4} \right)\| g\|^2\notag\\
&\leq  \frac{(6 N\delta)^4}{12\nu}\|Hg\|\|g\|  +\left(\frac{N\delta}{L_2} +\frac{1}{4} \right)\| g\|^2,
\end{align}
where we  made use of \eqref{eq:repartcons} in the last inequality.
Without loss of generality, we can assume that $\delta_0$ in Assumption \ref{a:lowerbddE} is small enough so that
 $\delta \leq L_2/(4N)$ for all $\delta\in (0,\delta_0)$. Thus
we can add \eqref{eq:estimateE} and \eqref{eq:estimateEc} to conclude that
\begin{align}
\|g\|\leq 4\left(\frac{1}{c_1 |k|}\left(\frac{L_2}{\delta}\right)^{m}+\tilde{c}_2\frac{\nu}{(|k| L^2_2)^2}\left(\frac{L_2}{\delta}\right)^{2m+4}+\tilde{c}_2\frac{\nu^{1/3}}{\left(|k|L_2\right)^{4/3}}\left(\frac{L_2}{\delta}\right)^{\frac{4}{3}(m+1)} +\frac{(6 N\delta)^4}{12\nu}\right)\| Hg\|.
\end{align}
We now take $\delta$ satisfying
\begin{align}
\frac{\delta}{L_2}=c_3 \left(\frac{\nu}{|k|}\right)^\frac{1}{m+4}
\end{align}
for some sufficiently small constant $c_3>0$. We conclude that
\begin{align}
   \nu^\frac{m}{m+4}|k|^\frac{4}{m+4} = \nu^\frac{m}{m+4}|k|^\frac{4}{m+4} \|g\|\leq c_4\| Hg\|,
\end{align}
for a large enough constant $c_4>0$ independent of $\nu$ and $k$. By definition then
\[
   \Psi(H)\geq \eps'_0\, \nu^\frac{m}{m+4}|k|^\frac{4}{m+4},
\]
for some constant $\eps'_0>0$ independent of $\nu$ and $k$, whence proving the proposition.

\end{proof}

\begin{remark}
The proof of Proposition \ref{t:SemigroupEst} carries over to the slightly more general case of the 
semigroup generated by the hypoelliptic operator
\begin{equation}
  \widetilde{H}_\nu=u(y)\de_x+\nu \de_y^4,
\end{equation}
for which the same semigroup estimate \eqref{bd:sembound} holds.
\end{remark}

\bibliographystyle{plain}
\bibliography{KSmixing}

\end{document}